\documentclass[11 pt]{article}
\usepackage{fullpage}
\usepackage{amsmath}
\usepackage{amssymb}
\usepackage{amsthm}
\usepackage{parskip}
\usepackage[all,cmtip]{xy}
\usepackage{xypic}
\newtheoremstyle{slplain}
  {.5\baselineskip\@plus.2\baselineskip\@minus.2\baselineskip}
  {.5\baselineskip\@plus.2\baselineskip\@minus.2\baselineskip}
  {\slshape}
  {}
  {\bfseries}
  {.}
  { }
  {}%
\theoremstyle{slplain}
\newtheorem{thm}{Theorem}[section]
\newtheorem{lem}[thm]{Lemma}

\theoremstyle{definition}
\newtheorem{defn}[thm]{Definition}

\theoremstyle{remark}
\newtheorem*{rem}{Remark}

\title{A Kadison Kastler row metric and intermediate subalgebras}
\author{Liam Dickson\thanks{Supported by an EPSRC doctoral training award}}
\AtEndDocument{\bigskip{\footnotesize%
  \textsc{Liam Dickson, Department of Mathematics, University of Glasgow, University Gardens, Glasgow G12 8QW, UK.} \par  
  \textit{E-mail address}: \texttt{liam.v.dickson@gmail.com}
}}
\begin{document}
\maketitle
\begin{abstract}
In this paper we introduce a row version of Kadison and Kastler's metric on the set of C*-subalgebras of $\mathbb{B}(\mathcal{H})$. By showing C*-algebras have row length (in the sense of Pisier) of at most 2 we show that the row metric is equivalent to the original Kadison Kastler metric.  Ino and Watatani have recently proved that in certain circumstances sufficiently close intermediate C*-algebras occur as small unitary perturbations. By adjusting their arguments to work with the row metric we are able to obtain universal constants independent of inclusions. 
\end{abstract}

\begin{section}{Introduction}
In \cite{Kad72} Kadison and Kastler initiated the study of the uniform perturbations of operator algebras. They introduced a metric on the set of linear subspaces of operators on a given Hilbert space, and conjectured that sufficiently close operator algebras  are necessarily spatially isomorphic. This conjecture was the focus of much research in the 1970's with seminal results being obtained by Christensen \cite{Chris77}, Johnson \cite{Johns77}, Phillips and Raeburn \cite{Rae77}, establishing a strong form of the conjecture in the case of injective von Neumann algebras: a von Neumann algebra $N$ sufficiently close to an injective von Neumann algebra $M$ arises as a small unitary perturbation of $M$, i.e. $N=uMu^*$ for a unitary with $u\approx 1$.

Analogous questions in the context of $C^*$-algebras have recently been the focus of considerable research activity, and remarkable progress has been made by Christensen et al. \cite{Chris10} using point-norm techniques to verify the conjecture in the separable nuclear case, vastly generalising the AF case established in \cite{Chris80} and the continuous trace case of \cite{PR81}.  Both separability and point norm techniques are necessary due to the counterexamples in the non-separable case of Choi and Christensen \cite{Choi83} and Johnson's construction of close representations of $C[0,1]\otimes\mathcal K$ which are not small unitary perturbations of each other \cite{Johns82}.  An embedding theorem for near inclusions involving separable nuclear $C^*$-algebras has also been obtained by combining these ideas with a strengthening of the completely positive approximation property, \cite{HKW12}.

The focus of this paper is on close subalgebras of a fixed $C^*$-algebra.  In \cite{Chris772} Christensen shows that sufficiently close von Neumann subalgebras of a finite von Neumann algebra arise from small unitary perturbations, and his work gives uniform estimates valid for all finite von Neumann algebras. In their recent paper Ino and Watatani \cite{Wat} prove an analogous theorem in  the context of intermediate $C^*$-subalgebras \cite[Proposition 3.6]{Wat}: given an inclusion of C*-algebras $C \subseteq D$ with a finite index conditional expectation, and intermediate subalgebras $A,B$ with $C\subseteq A,B\subseteq D$, if $A$ and $B$ are sufficiently close then they are necessarily small unitary perturbations of each other.  In contrast to Christensen's work, the estimates obtained by Ino and Watatani are given in terms of the basis for $C\subseteq D$ and so depend on the inclusion $C\subseteq D$.  The main purpose of this paper (Theorem \ref{variant}) is to obtain uniform estimates valid for all finite index inclusions.

The main ingredient in the proof of Theorem \ref{variant} is to work with a ``row'' version of the Kadison-Kastler metric.  Natural variants of the metric have been considered previously. In particular there is a completely bounded version of the metric, where the distance between $A$ and $B$ is obtained as the supremum of the distances between the matrix amplifications $M_n(A)$ and $M_n(B)$. There is a deep connection between the Kadison-Kastler perturbation conjecture and Kadison's similarity problem (which asks whether every bounded unital homomorphism $A\rightarrow\mathbb B(H)$ from a $C^*$-algebra is similar to a $^*$-homomorphism) dating back to \cite{Chris80} which is exemplified by the characterisation that the similarity problem is true for all $C^*$-algebras if and only if the completely bounded Kadison-Kastler metric is equivalent to the Kadison-Kastler metric, \cite{CCSSWW13,Chris10}.  The row metric naturally fits between the usual metric and the completely bounded metric, and the main technical observation, which is of interest in its own right, is that the row metric is equivalent to the Kadison-Kastler metric.  With this it is shown that Ino and Watatani's techniques can be performed working with row metric estimates leading to universal constants.

To establish the equivalence of the row metric and the Kadison-Kastler metric we use Haagerup's Little Groethendick Inequality (\cite[Lemma 3.2]{Haag83c}) as an ingredient in Pisier's intrinsic characterisation of the similarity property in terms of matrix factorisation length \cite{Pis98}.  This enables us to demonstrate that rows over $C^*$-algebras can always be factorised in a uniform fashion, with factorisation length at most $2$. The equivalence of two metrics then follows using an argument from \cite{PisPis01}.

The paper is structured as follows: firstly, we prove the existence of length 2 factorisations for rows over unital C*-algebras from which we deduce the equivalence of the original Kadison- Kastler metric and the row metric. In the second section we modify the techniques of Ino and Watatani to work with the row metric.

\end{section}

\begin{section}{Row Factorisations and the Row Metric}

We recall the definition of the Kadison Kaslter metric from \cite{Kad72}, which measures the Hausdorff distance between the unit ball of two C*-algebras, and the completely bounded metric  (see \cite{Chris10} and \cite{Chris12} for properties of $d_{\text{cb}}$).
\begin{defn}
Let $C$ be a C*-algebra with closed linear subspaces $A$ and $B$. We define the distance $d(A,B)$ from $A$ to $B$  as
\begin{equation}
d(A,B) = \max  \{ \sup_{y \in B_1} \inf_{x \in A_1} \| x-y \| , \sup_{x \in A_1} \inf_{y \in B_1} \| x - y \| \}
\end{equation}
and the completely bounded distance
\begin{equation}
d_{\text{cb}}(A,B) = \sup_{n \in \mathbb{N}} d(M_n(A), M_n (B) ).
\end{equation}
\end{defn}
We now introduce an intermediate metric on the set of C*-subalgebras by taking the supremum of the distance over all row amplifications. We will see (Corollary \ref{rm}) that this metric is, in fact, equivalent to $d$.
\begin{defn}
Let $C$ be a C*-algebra with closed linear subspaces $A$ and $B$. Write
\begin{equation}
d_{\text{row}}(A,B) = \sup_{n \in \mathbb{N}} d(M_{1,n}(A), M_{1,n}(B)),
\end{equation}
where $M_{1,n}(A)$ and $M_{1,n}(B)$ can be thought of as closed linear subspaces of the C*-algebra $M_n(C)$.
\end{defn}

Given C*-algebras $A$ and $B$ and a linear map $\phi:A \rightarrow B$ for $m, n \in \mathbb{N}$, write
\begin{equation}
\phi^{(m,n)}((x_{ij})_{ij}) = (\phi(x_{ij}))_{ij}, \qquad (x_{ij})_{ij} \in M_{m,n}(A).
\end{equation}
We will use the shorthand $\phi^{(m)}$ to denote $\phi^{(m,m)}$. We define the row norm of a linear map by taking the supremum of row amplifications
\begin{equation}
\|\phi \|_{\text{row}} = \sup_{n \in \mathbb{N}} \| \phi^{(1,n)} \|.
\end{equation}


The definition below is due to Pisier \cite{Pis98}. Finite length reflects the ability to factorise elements  in all matrix amplifications of the algebra into diagonal matrices and scalar matrices, simultaneously controlling the length and the norm of factorisations. The final definition of row length only requires the existence of uniform factorisations, as above, for row amplifications of the algebra. Finite row length is, therefore, a weaker condition than finite length.
\begin{defn}\label{factordef}
An operator algebra $A$ has length at most $d$ if there exists a positive constant $K$ such that for every $n, m \in \mathbb{N}$ and $x \in M_{m,n} (A)$ there exists an integer $N \in \mathbb{N}$ and matrices $C_1, \dots, C_{d+1}, D_1, \dots, D_d$ where;
\begin{itemize}
\item $C_1 \in M_{m, N} (\mathbb{C})$,
\item $C_k \in M_N (\mathbb{C})$ for $2 \leq k \leq d$,
\item $C_{d+1} \in M_{N,n}(\mathbb{C})$,
\item $D_l $ are diagonal matrices with entries in the unit ball of $A$ for $1 \leq l \leq d$
\end{itemize}
which satisfy
\begin{equation}
x = C_1 D_1C_2 D_2 \dots C_d D_d C_{d+1} 
\end{equation}
and
\begin{equation}
\prod_{i=1}^{d+1} \| C_i \| \leq K \|x \|.
\end{equation}
If no such constants $d$ and $K$ exist we say that $A$ has inifinite length.

If we only insist that for every $n \in \mathbb{N}$ every element of $M_{1,n} (A)$ can be factorised as above then we say the operator algebra $A$ has row length at most $d$.

To encapsulate these factorisations for each positive integer $d$, define a norm $\| \cdot \|_{(d)}$ on matrix amplifications of $A$ as follows, for $x \in M_{m,n}(A)$ set
\begin{equation}
\| x \|_{(d)} = \inf \{ \prod_{i=1}^{d+1} \| C_i \|  : x = C_1 D_1C_2 D_2 \dots C_d D_d C_{d+1} \}
\end{equation}
where the infimum is taken over all factorisations into scalar and diagonal matrices as described above.
\end{defn}
We now recall the definition of the maximal operator space (see Chapter 14 of \cite{Pauls02}).
\begin{defn}
Let $V$ be a normed vector space, the maximal norm $\| \cdot \|_{\text{MAX(V)}}$ is defined for $m, n \in \mathbb{N}$ and $(x_{ij})_{ij} \in M_{m,n}(V)$ by
\begin{equation}
\| (x_{ij})_{ij} \|_{\text{MAX}(V)} = \sup \{ \| (\phi(x_{ij}))_{ij} \|_{M_{m,n}(\mathbb{B} ( \mathcal{H}))} \},
\end{equation}
where the supremum is taken over all Hilbert spaces $\mathcal{H}$ and linear isometries $\phi: V \rightarrow \mathbb{B}(\mathcal{H})$.

For use in the sequel we define a scaling of the maximal operator space. For $c>0$ define $\text{MAX}(V)_c$ to be the operator space $\text{MAX}(V)$ with norms $\| \cdot \|_{\text{MAX}(V)_c}$ defined as follows; for $m, n \in \mathbb{N}$ and $x \in M_{m,n}(V)$ then
\begin{equation}
\| x \|_{\text{MAX}(V)_c}  = c \| x \|_{\text{MAX}(V)}.
\end{equation}
\end{defn}

Recall that we may always view an abstract operator space as a closed subspace of the universal operator algebra as set out below (see Chapter 19 of \cite{Pauls02} for more details).
\begin{defn}\label{universal}
Let $V$ be a normed vector space and let $\mathcal{F}(V) = \sum_{n=0}^\infty\oplus V^{\otimes n}$  be the vector space of finite direct sums of elements from the algebraic tensor powers of $V$. Define a multiplication $\odot$ on $\mathcal{F}(V)$ as follows; if $u = u_1 \otimes \dots \otimes  u_k $ and $v = v_1 \otimes \dots \otimes v_l$ are elementary tensors in $\mathcal{F}(V)$ then
\begin{equation}
u \odot v =(u_1 \otimes \dots \otimes u_k) \odot (v_1 \otimes \dots \otimes v_l) = u_1 \otimes \dots \otimes u_l \otimes v_1 \otimes \dots \otimes v_l.
\end{equation}
\end{defn}
The above multiplication gives $\mathcal{F}(V)$ an algebra structure and satisfies the following universal property: if $\mathcal{A}$  is an algebra and $\phi: V \rightarrow \mathcal{A}$ a linear map then there exists a unique algebra homomorphism $\pi_\phi: \mathcal{F}(V) \rightarrow \mathcal{A}$. It is defined on elementary tensors $u_1 \otimes \dots \otimes u_k$ as
\begin{equation}
\pi_\phi(u_1 \otimes \dots \otimes u_k) = \phi(u_1) \dots \phi(u_k).
\end{equation}
Now if $V$ is also an operator space we may define the norm $\| \cdot \|_1$ on all matrix amplifications of $\mathcal{F}(V)$ as follows, if $m,n \in \mathbb{N}$ and  $x = (x_{ij})_{ij} \in M_{m,n}(\mathcal{F}(V))$ then
\begin{equation}
\| x \|_1 = \sup \{ \| (\pi_\phi(x_{ij}))_{ij} \| :  \| \phi \|_{\text{cb}} \leq 1 \},
\end{equation}
where the supremum is taken over all Hilbert spaces $\mathcal{H}$ and all linear maps $\phi : V \rightarrow \mathbb{B}( \mathcal{H})$ with $\|\phi\|_{cb} \leq 1$. Let $\text{OA}_1 (V)$ be the operator algebra obtained by completing $\mathcal{F}(V)$ with respect to the norm $\| \cdot \|_1$. 

Furthermore $\text{OA}_1 (V)$ has the property that if $A$ is an operator algebra and $\phi: V \rightarrow A$ is completely contractive then $\pi_\phi$ extends to a completely contractive map $\pi_{\phi}: \text{OA}_1 (V) \rightarrow A$.

Haagerup showed in \cite[Theorem 1.10]{Haag83c} (this was also independently obtained by Hadwin \cite{Had81} and Wittstock \cite{Wit81}) that a bounded non-degenerate homomorphism from a C*-algebra on a Hilbert space $\mathcal{H}$ is completely bounded if and only if it is similar to a *-representation. Therefore, the similarity problem may be phrased for unital C*-algebras by asking  whether every bounded unital homomorphism into $\mathbb{B}(\mathcal{H})$ is completely bounded.   In \cite{Pis98}  Pisier shows that the latter property is equivalent to finite length which, therefore, provides an intrinsic characterisation of the similarity property.  That all unital C*-algebras have row length at most 2 follows implicitly from \cite[Theorem 2.6]{Pis98} and Lemma \ref{rowbound} below. However, we include a proof for the reader's convenience and in order to obtain an explicit constant. The proof is a modification of Paulsen's exposition (\cite{Pauls02}) of Pisier's results in \cite{Pis98}.
\begin{thm}\label{rowfactor}
Let $A$ be a unital C*-algebra, then $A$ has row length at most 2 with
\begin{equation}
\|x\|_{(2)} \leq \inf_{c > \sqrt{2} +1} \left( \frac{\sqrt{2}(c^3 -1)}{c-\sqrt{2}-1} \right) \|x\|<55\|x\|,
\end{equation}
for $x \in M_{1,n}(A)$ for some $n \in \mathbb{N}$.
\end{thm}
Firstly, we extract a lemma from Paulsen's exposition \cite[Chapter 19]{Pauls02} of Pisier's theorem for use in the sequel. It allows us to factorise elements of an operator algebra $A$ provided we can lift certain elements under the maps described below.

Fix $c > 1$ and let $\iota: \text{MAX}(A)_c \rightarrow A$ be the inclusion map. For $m, n \in \mathbb{N}$ and $x \in M_{m,n} (A)$ we  have
\begin{equation}
\|x \|_{\text{MAX}(A)_c} = c\| x \|_{\text{MAX}(A)} \geq  \| x \|_{\text{MAX}(A)} \geq \|x\|
\end{equation}
and hence $\iota$ is completely contractive. Therefore $\iota$ extends to a homomorphism
\begin{equation}
\pi_{\iota, c} : \text{OA}_1 ( \text{MAX} (A)_c ) \rightarrow A
\end{equation}
which satisfies $\pi_{\iota,c}(a) = a$ for $a \in A \subseteq \text{OA}_1 ( \text{MAX} (A)_c )$ and is a complete contraction by the universal property described in the paragraph following Definition \ref{universal}.

\begin{lem}\label{factor}
Let $M > 0, c>1$ and $A$ be a unital operator algebra. For $m, n \in \mathbb{N}$  suppose that for each element $a \in M_{m,n}(A)$ with $\| a \| \leq 1/M$ there exists an element $y \in M_{m,n}(\mathcal{F} (\text{MAX}(A)_c))$ in the open unit ball of $\text{OA}_1(\text{MAX}(A)_c)$ satisfying $\pi_{\iota,c}^{(m,n)}(y) = a$. If $d$ is an integer satisfying $ c^d(c-1)> M$  then each element in $ x \in M_{m,n}(A)$ satisfies
\begin{equation}
\| x \|_{(d)} \leq \frac{M(c^{d+1}-1)}{c^{d+1}-c^d - M}\|x\|.
\end{equation}
\end{lem}
\begin{proof}
The proof of this is the last sentence of the second paragraph and final paragraph of \cite[Proposition 19.9]{Pauls02} followed by the proof of \cite[Theorem 19.11]{Pauls02}.
\end{proof}
The next step is to establish that a bounded homomorphism from a unital C*-algebra into $\mathbb{B}(\mathcal{H})$ is automatically bounded in the row norm. This follows from Haagerup's Little Groethendick inequality, stated below.
\begin{lem}[Haagerup, Lemma 3.2  of \cite{Haag83c}]\label{cyclic}
Let $A$ be a C*-algebra, let $\mathcal{H}$ be a Hilbert space, and let $T: A \rightarrow \mathcal{H}$ be a bounded linear map. Then there exist two states $f$ and $g$ on $A$, such that 
\begin{equation}
\|T(x)\|^2 \leq \|T\|^2(f(x^*x) + g (x x^* )), \qquad x \in A.
\end{equation}
\end{lem}
The next lemma is a direct consequence of Lemma $\ref{cyclic}$. The author would like to thank Erik Christensen for suggesting this estimate which improves those used in an the earlier version of this work.
\begin{lem}\label{rowbound}
Let $A$ be a unital C*-algebra then every unital bounded homorphism $\phi: \mathcal{A} \rightarrow \mathbb{B} ( \mathcal{H})$ satisfies
\begin{equation}
\| \phi \|_{\text{row}} \leq \sqrt{2} \|\phi\|^2.
\end{equation}
\end{lem}
\begin{proof}
Let $x \in A$ have polar decomposition $x = hv$  in $A^{**}$ (with $h = (xx^*)^{1/2}$). 
Let $\bar{\phi} : A^{**} \rightarrow \mathbb{B}(\mathcal{H})^{**}$ denote the extension of $\phi$ to the bi-dual of $A$, so that $\phi(x) = \phi(h)\bar{\phi}(v)$. Hence
\begin{equation}
 \phi(x)\phi(x)^* = \phi(h)\bar{\phi}(v)\bar{\phi}(v)^* \phi(h)^* \leq \| \phi \|^2 \phi(h)\phi(h)^*.
\end{equation}
 Let $\xi \in \mathcal{H}$ be a unit vector. Applying  Lemma \ref{cyclic} to the bounded linear map $x \mapsto \phi(x^*)^*\xi$ gives states $f$ and $g$ on $A$, such that
\begin{equation}
\|\phi(x)^* \xi \|^2 \leq \|\phi \|^2 \|\phi(h)^* \xi \|^2 \leq \|\phi\|^4(f(h^2) + g(h^2)) =  \|\phi\|^4 (f(xx^*) + g (xx ^*)).
\end{equation}
Fix $n \in \mathbb{N}$, let $x = (x_1, \dots, x_n) \in M_{1,n}(A)$  then, 
\begin{equation}
\|\phi^{(1,n)}(x)^*\xi \|^2  = \sum_{j=1}^n \|\phi(x_j)^*\xi \|^2 \leq \|\phi\|^4 \sum_{j=1}^n (f(x_jx_j^*) + g (x_j x_j^*)) \leq 2 \|\phi\|^4 \|x x^* \|= 2\|\phi\|^4\|x\|^2.
\end{equation}
It follows that $\| \phi^{(1,n)}(x) \| = \| \phi^{(1,n)}(x)^* \| \leq \sqrt{2} \| \phi \|^2 \|x\|$.
\end{proof}

The previous result provides the ingredient needed to apply Lemma \ref{factor} to rows of a unital C*-algebra. The proof follows the first two pages of \cite[Proposition 19.9]{Pauls02}.

\begin{lem}\label{rowlift}
Let $c> 1$ and $A$ be a unital C*-algebra then for each $n \in \mathbb{N}$ and element $a \in M_{1,n}(A)$ satisfying $\| a \| < \frac{1}{\sqrt{2}c^2}$ there exists an  element $y \in M_{1,n}(\mathcal{F} (\text{MAX}(A)_c)) $ in the open unit ball of $\text{OA}_1(\text{MAX}(A)_c)$ satisfying $\pi_{\iota,c}^{(1,n)}(y) = a$.
\end{lem}

\begin{proof}
Using the notation of the paragraph preceding Lemma \ref{factor},
the restriction $\pi_{\iota,c} |_{\mathcal{F}(\text{MAX} (A)_c)} : \mathcal{F}(\text{MAX} (A)_c) \rightarrow A$ is completely contractive.

Let $I = \text{ker}( \pi_{\iota,c}) \cap  \mathcal{F}(\text{MAX} (A)_c)$ which is a closed ideal in $ \mathcal{F}(\text{MAX} (A)_c)$.  The algebra $ \mathcal{F}(\text{MAX} (A)_c)$ inherits the operator space structure from $\text{OA}_1( \text{MAX} (A)_c )$ and hence is a  non-complete operator algebra and so is the  quotient $Z :=  \mathcal{F}(\text{MAX} (A)_c) / I$. 

As noted in the proof of Theorem 1.7 \cite{Pis98}, the completion $\tilde{Z}$ inherits operator space structure induced by $\text{OA}_1 ( \text{MAX} (A)_c )$ and the multiplication is completely contractive so is an abstract operator algebra.  The map induced by $\pi_{\iota,c}$,
\begin{equation}
\tilde{\pi} : \tilde{Z} \rightarrow A
\end{equation}
is a completely contractive homomorphism and is onto and injective when restricted to $Z$. Let  the homomorphism $\rho: A \rightarrow \tilde{Z}$ be the inverse of $\tilde{\pi}|_Z$. 
Recall, for $a \in A$, we have $\pi_{\iota,c}(a) =a$, therefore
\begin{equation}
\rho (a) = \rho (\pi_{\iota,c} (a) ) = \rho (\tilde{\pi} (a + I)) = a + I
\end{equation}
and so
\begin{equation}
\| \rho(a) \|= \| a + I \|_{\tilde{Z}} \leq \| a \|_c = c \| a \|.
\end{equation}
Therefore the unital homorphism $\rho$ satisfies $\| \rho \| \leq c$.
Now let $n \in \mathbb{N}$ be fixed and let $a = (a_1, \dots a_n) \in M_{1,n } (A)$ such that $\| a \| < \frac{1}{\sqrt{2} c^2}$.
By the BRS theorem (which first appeared in \cite{Blech90}, see  \cite[Corollary 16.7]{Pauls02} for the formulation we use here) we may assume that $\tilde{Z}$ is isometrically represented on $\mathbb{B}( \mathcal{H})$ for some Hilbert space $\mathcal{H}$ and so we may apply Lemma \ref{rowbound} to yield
\begin{equation}
\| \rho^{(1,n)} (a) \| \leq \sqrt{2} c^2 \| a\| < 1.
\end{equation}
On the other hand
 \begin{equation}
\rho^{(1,n)} (a) = ( \rho(a_1), \dots, \rho(a_n) ) = (a_1 + I, \dots, a_n + I) \in M_{1,n}(Z),
\end{equation}
so we may find elements $z_1, \dots, z_n \in I$ such that 
\begin{equation}
\| (a_1+ z_1, \dots, a_n + z_n) \|_{\text{OA}_1(\text{MAX} ( A)_c)} < 1.
\end{equation}
Let $y= (a_1+ z_1, \dots, a_n + z_n) \in M_{1,n} (\mathcal{F} ( \text{MAX} (A)_c)$. Then $\pi_{\iota,c}^{(1,n)}(y) = a$ as required.
\end{proof}
To complete the proof of Theorem \ref{rowfactor} we firstly restrict ourselves to row elements of small radius. We use Lemma \ref{rowlift} to find a lift into the universal operator algebra and then directly apply Lemma \ref{factor} to obtain  row factorisations.
\begin{proof}[Proof of Theorem \ref{rowfactor}]
Fix $n \in \mathbb{N}$ and $c > \sqrt{2}+1$. Set $M = \sqrt{2}c^2$  and let $a \in M_{1,n} (A)$ such that $\| a \| < \frac{1}{M}$. Applying Lemma \ref{rowfactor} we may find an element $y \in M_{1,n}(\mathcal{F} (\text{MAX}(A)_c)) $ in the unit ball of $\text{OA}(\text{MAX}(A)_c)$ satisfying $\pi_{\iota,c}^{(1,n)}(y) = a$. Since $c>\sqrt{2}+1$ implies that 
\begin{equation}
c^2(c-1) > \sqrt{2}c^2 = M,
\end{equation}
the hypothesis of Lemma \ref{factor} are satisfied with $d =2$ and so any element of $x \in M_{1,n} (A)$ satisfies
\begin{equation}
\| x \|_{(2)} \leq \frac{M(c^3 -1 )}{c^3- c^2 - M}\|x\| =\frac{\sqrt{2}(c^3-1)}{c-\sqrt{2}-1}\|x\|,
\end{equation}
for any $c > \sqrt{2}+1$. 
Since $n$ was arbitrary the proof is complete.
\end{proof}

With this factorisation in hand, we are in a position to translate bounds involving the original Kadison Kastler metric to bounds in the row metric for unital C*-algebras. 
\begin{thm}\label{rm}
The metrics $d$ and $d_{row}$ are equivalent on unital C*-algebras. In particular, if $A,B \subseteq \mathbb{B} (\mathcal{H})$ are unital C*-algebras, the following inequality holds:
 \begin{equation}
d_{row}(A,B) \leq 220 d(A,B).
\end{equation}
\end{thm}
The proof is a modification of \cite[Proposition 2.10]{Chris12}, which was first observed by Pisier in the remark following Theorem 10.13 in \cite{PisPis01}.
\begin{proof}
Fix $m \in \mathbb{N}$ and let $x$ be in the unit ball of $M_{1,m}(A)$. Since 
\begin{equation}
\inf_{c> \sqrt{2}+1} \left(\frac{\sqrt{2}(c^3-1)}{c-\sqrt{2}-1}\right)< 55
\end{equation}
we may apply Theorem \ref{rowfactor} to find a factorisation
\begin{equation}
 x= C_1D_1C_2D_2C_3
\end{equation}
with $N \in \mathbb{N}$;  $C_1 \in M_{1,N}(\mathbb{C})$, $C_2 \in M_{N}(\mathbb{C})$ and $C_3 \in M_{N,m}(\mathbb{C})$  scalar matrices satisfying
\begin{equation}
\prod_{i=1}^{3} \| C_i \| \leq  55
\end{equation}
and $D_1, D_2$ are diagonal matrices with entries $D_i^{(j)}$ in the unit ball of $A$ for $1 \leq j \leq N$. For each $i=1,2$ and  $1 \leq j  \leq N$ let $E_i^{(j)}$ be an element of the unit ball of $B$ such that $\|D_i^{(j)}-E_i^{(j)} \| \leq \gamma$ using the hypothesis $d(A,B) \leq \gamma$. Let $E_i$ be the diagonal matrix in $M_N(B)$ with $E_i^{(j)}$  in the $(j,j)$ entry. Then for $i=1,2$ we have
\begin{equation}
\|D_i - E_i \| \leq \gamma.
\end{equation}
By construction the element 
\begin{equation}
y' = C_1 E_1 C_2 E_2 C_3 .
\end{equation}
is in $M_{1,m}(B)$. Furthermore
\begin{equation}
\| x-y' \| \leq \| C_1( D_1-E_1)C_2 D_2 C_3 \| + \| C_1 E_1 C_2 (D_2 - E_2)C_3\| \leq 110 \gamma
\end{equation}
Finally the element $y = y' / \| y' \|$ is in the unit ball of $M_{1,m}(B)$ and satisfies $\| x-y \|  \leq 220 \gamma$. The same argument may be repeated to approximate elements in the unit ball of $M_{1,m}(B)$ with those in $M_{1,m}(A)$ and, since $m$ was arbitrary, the bound is as claimed.
\end{proof}
\begin{rem}
Automatic row closeness also follows in the non-unital case. To see this suppose that $A$ and $B$ are non-unital C*-subalgebras of a C*-algebra $C$ with $d(A,B)$  small. Let $\tilde{C}= C \oplus \mathbb{C}$ then for the unital C*-subalgebras $\tilde{A}= A \oplus \mathbb{C}$ and $\tilde{B}= B \oplus \mathbb{C}$ of $\tilde{C}$ we will have $d(\tilde{A}, \tilde{B})$ small. It now follows from Theorem \ref{rm} that $d_{\text{row}}(\tilde{A}, \tilde{B})$ will be small and since the quotient map $\pi:\tilde{C} \rightarrow C$ is row contractive we have $d_{\text{row}}(A,B)$ small.
\end{rem}
\end{section}

\begin{section}{Universal Constants for Ino and Watatani's Theroem}
The C*-basic construction is studied in detail in \cite{Wat90}, it provides a C*-analogue of techniques developed by Jones in his work on subfactors of von Neumann algebras  \cite{Jones83}.We recall some details, starting with the definition of the index of a conditional expectation.
\begin{defn}
Let $C \subseteq D$ be C*-algebras. A conditional expectation $E: D \rightarrow C$ is of finite index if there exists a finite set $v_1, \dots, v_n \in D$ such that 
\begin{equation}\label{watdef}
x = \sum_{i=1}^n v_i E(v_i^* x)
\end{equation}
for all $x \in D$.
The set $\{v_1, \dots, v_n \}$ is called a quasi-basis for $E$ and the index of $E$ is defined by  $\sum_{i=1}^n v_i v_i^*$.
\end{defn}
The index is independent of the choice of quasi-basis and,  furthermore, it is an invertible, central element of $D$ \cite[Proposition 1.2.8 and Lemma 2.3.1]{Wat90}.

Suppose $B \subseteq D$ is an inclusion of C*-algebras with a faithful conditional expectation  $E_B: D \rightarrow B$. We  assign a sesquilinear form to $D$ as follows. For $x, y \in D$ write
\begin{equation}
\langle x, y \rangle_B = E_B(x^* y).
\end{equation}
The completion $\mathcal{E}$ of $D$ with respect to the norm $\|\cdot\|_B = \|\langle\cdot, \cdot \rangle_B\|^{1/2}$ is a Hilbert $B$-module when equipped with the inner product above. Let $\eta:D \rightarrow \mathcal{E}$ be the natural inclusion map which is injective as $E_B$ was assumed to be faithful. Let $\mathbb{B}(\mathcal{E})$  denote the C*-algebra of adjointable operator on $\mathcal{E}$.

 The Jones projection $e_B \in \mathbb{B}(\mathcal{E})$ is defined by extending
\begin{equation}\label{Jones}
e_B ((\eta(x)) = \eta( E_B (x)), \qquad x \in D
\end{equation}
to $\mathcal{E}$ by continuity.
 
The left regular representation is given by the *-homomorphism $\lambda: D \rightarrow \mathbb{B} (\mathcal{E})$ 
 where, for $x \in D$, $\lambda(x)$ is defined by extending
\begin{equation}
\lambda(x) ( \eta (y) ) = \eta(xy), \qquad y \in D
\end{equation}
to $\mathcal{E}$ by continuity.
We recall the following facts relating the left regular representation and the Jones projection \cite[Lemma 2.1.1.]{Wat90}.
\begin{lem}\label{covariant}
With $E_B, e_B$ and $\lambda$ as above we have;
\begin{enumerate}
\item for all  $x \in D$ we have $\lambda(x) e_B = e_B\lambda (x) $ if and only if $x \in B$,
\item $e_B \lambda(x) e_B = \lambda(E_B(x))e_B$ for all $x \in D$,
\item $x \mapsto \lambda(x)e_B$ is an isomorphism of $B$ into $\mathbb{B} (\mathcal{E})$.
\end{enumerate}
\end{lem}

We now provide row versions of the estimates from \cite{Wat} starting with a `row version' of \cite[Lemma 3.2]{Wat}.
\begin{lem}\label{watlem}
Suppose that $A$ and $B$ are C*-sublagebras of a C*-algebra $D$ and suppose that $E_B:D \rightarrow B$ is a conditional expectation. Let $\iota_A: A \rightarrow D$ be the inclusion map, then
\begin{equation}\label{ceid}
\|E_B|_A - \iota_A\|_{\text{row}} \leq 2 d_{\text{row}}(A,B).
\end{equation}
Furthermore, for $m \in \mathbb{N}$ and an element $x$ in the unit ball of $M_{1,m}(A)$ we have
\begin{equation}\label{ceid2}
\|E_B^{(1,m)}(x)E_B^{(m,1)}(x^*)-E_B(xx^*) \| \leq 4 \gamma
\qquad \text{and} \qquad 
\|E_B^{(m,1)}(x^*)E_B^{(1,m)}(x)-E_B^{(m)}(x^*x) \| \leq 4 \gamma.
\end{equation}
\end{lem}
\begin{proof}
Let $\epsilon > 0$ and set $\gamma = d_{\text{row}}(A,B)+ \epsilon$. For $m \in \mathbb{N}$ and  $x$ in the unit ball of $ M_{1,m}(A)$, there exists $x' $ in the unit ball of $M_{1,m}(B)$ such that 
$\|x-x'\| \leq \gamma$.
We have
\begin{equation}
\|E_B^{(1,m)}(x) - x \| \leq \|E_B^{(1,m)}(x - x') \| + \|x'-x\| \leq 2 \gamma,
\end{equation}
since $\epsilon$ was arbritrary this proves \eqref{ceid}, and
\begin{equation}\label{wattri}
\|xx^*-x'x'^{*}\| \leq \|(x-x')x^*\|+ \|x'(x^*-x'^*)\| \leq 2\gamma.
\end{equation}
Since $x', x'^* $ and $x'x'^*$ are in $M_{1,m}(B)$, $M_{m,1}(B)$ and $B$ respectively, we have  $E_B(x'x'^*) = x'x'^*=E_B^{(1,m)}(x')E_B^{(m,1)}(x'^*)$. Combining this with \eqref{wattri} yields
\begin{align}
\|E_B^{(1,m)}(x)E^{(m,1)}_B(x^*) - E_B(xx^*)\|  & \leq \| E_B^{(1,m)}(x)(E_B^{(m,1)}(x^*)-E_B^{(m,1)}(x'^*))\| \nonumber \\
                                                                             &  \qquad +\|((E_B^{(1,m)}(x)-E_B^{(1,m)}(x'))E_B^{(m,1)}(x'^*)\| \nonumber \\
                                                                              & \qquad + \|E_B(x'x'^*)-E_B(xx^*)\| \leq 4\gamma.
\end{align}
As above this establishes \eqref{ceid2} since $\epsilon$ was arbitrary.  The final estimate follows in a similar fashion.
\end{proof}
Next we show how statements about the approximate multiplicativity of the conditional expectation can be translated to statements about the norm of operators in $\mathbb{B}(\mathcal{E})$.
\begin{lem}\label{watlem2}
Suppose that $D$ is a C*-algebra and $B$ is a C*-subalgebra with faithful conditional expectation 
$
E_B: D \rightarrow B
$
with $e_B \in \mathbb{B}(\mathcal{E})$ as defined above.
Let $m \in \mathbb{N}$ and $x = (x_1, \dots, x_m)  \in M_{1,m}(D)$. Then the following identities hold
\begin{equation*}
\|E_B^{(1,m)}(x)E_B^{(m,1)}(x^*)-E_B(xx^*) \| =  \|e_B  \lambda^{(1,m)}(x)(\text{diag}^{(m)}(1_{\mathcal{E}}-e_B))   \lambda^{(m,1)}(x^*)e_B  \|_{ \mathbb{B}( \mathcal{E})}
\end{equation*}
and
\begin{equation*}
\|E_B^{(m,1)}(x^*)E_B^{(m,1)}(x)-E_B^{(m)}(x^*x) \|_{M_m(D)} =  \|\text{diag}^{(m)}(e_B)  \lambda^{(m,1)}(x^*)(1_{\mathcal{E}}-e_B)   \lambda^{(1,m)}(x)\text{diag}^{(m)}( e_B ) \|_{ M_m( \mathbb{B}( \mathcal{E}))}
\end{equation*}

\end{lem}
\begin{proof}
Since the map $b \mapsto \lambda(b)e_B$ is a *-isomorphism so is its amplification 
\begin{equation}
(x_{ij})_{ij} \mapsto ( \lambda(x_{ij})e_B )_{ij}
\end{equation}
 which takes $ M_{m}(B)$ into $M_m(\mathbb{B}(\mathcal{E})) $. Writing $x =(x_1, \dots, x_m)$  we use the observation in the previous sentence and condition 1 of Lemma \ref{covariant} to compute
\begin{align}
\|E_B^{(1,m)}(x)E_B^{(m,1)}(x^*)-E_B(xx^*)\| & =  \| \lambda(E_B^{(1,m)}(x)E_B^{(m,1)}(x^*)-E_B(xx^*))e_B \|_{\mathbb{B}(\mathcal{E})}\nonumber  \\
                                                                           & = \| \sum_{j=1}^m \lambda(E_B(x_j)E_B(x_j^*) - E_B(x_j x_j^*))e_B\|_{\mathbb{B}(\mathcal{E})} \nonumber \\
                                                                          & = \| \sum_{j=1}^m e_B\lambda(x_j)e_B \lambda(x_j^*)e_B - e_B\lambda(x_j x_j^*)e_B\|_{\mathbb{B}(\mathcal{E})} \nonumber \\
                                                                         &=\| \sum_{j=1}^me_B\lambda(x_j)(1_{\mathcal{E}}-e_B)  \lambda(x_j^*)e_B \|_{\mathbb{B}(\mathcal{E})} \nonumber \\
                                                                         & =  \|e_B \lambda^{(1,m)}(x) (\text{diag}^{(m)}(1_{\mathcal{E}}-e_B)) \lambda^{(m,1)}(x^*)  e_B\|_{\mathbb{B}(\mathcal{E})}
\end{align}
and
\begin{align}
\|E_B^{(m,1)}(x^*)E_B^{(1,m)}(x)-E_B^{(m)}(x^*x) \|_{M_m(D)} & =  \| ( E_B(x_i^*)E_B(x_j) -E_B(x_i^* x_j))_{ij}  \|_{ M_m(D)} \nonumber \\
                                                                                            & =  \|(\lambda( E_B(x_i^*)E_B(x_j) -E_B(x_i^* x_j))e_B)_{ij} \|_{ M_m( \mathbb{B}( \mathcal{E}))} \nonumber \\
                                                                                            &=  \|(e_B(\lambda(x_i^*) (1_{\mathcal{E}}-e_B)\lambda(x_j) e_B)_{ij} \|_{ M_m( \mathbb{B}( \mathcal{E}))} \nonumber \\
                                                                                            & =  \|\text{diag}^{(m)}(e_B)  \lambda^{(m,1)}(x^*)(1_{\mathcal{E}}-e_B)   \lambda^{(1,m)}(x)\text{diag}^{(m)}(e_B ) \|_{ M_m( \mathbb{B}( \mathcal{E}))}.
\end{align}
\end{proof}

We now modify  \cite[Lemma 3.4]{Wat} to work with the row metric obtaining universal constants independent of the inclusion $C \subseteq D$. The argument is based on techniques developed by Christensen in \cite{Chris77} and  \cite{Chris772}.
\begin{lem}\label{homo}
Let $C \subseteq D$ be a unital inclusion of C*-algebras.
 Suppose that $B \subseteq D$ is a C*-algebra containing $C$ such that there exists a faithful conditional expectation $E_B^D: D \rightarrow B$. Suppose that $A \subseteq D$ is another C*-algebra  containing $C$ with a finite index conditional expectation $E_C^A: A \rightarrow C$ such that $d_{\text{row}}(A,B) \leq \gamma < 1/16$.  Let $\iota_A: A \rightarrow D$ denote the inclusion map. Then there exists a *-homomorphism
$ \phi: A \rightarrow B$ such that $\| \phi- \iota_{A}\|_{\text{row}} \leq 8 \sqrt{2} \gamma^{\frac{1}{2}} + 2\gamma $ and  $\phi|_C = \text{id}_c$.
\end{lem}
\begin{proof}
Let $\mathcal{E}$ be the completion of $D$  with the norm derived from $E_B^D$ as described in the paragraph preceding Lemma \ref{covariant} and Jones projection $e_B \in \mathbb{B}(\mathcal{E})$. 
Let $(v_i)_{i=1}^n$ be a quasi-basis for $E_C^A$ in $A$
with $T= \sum_{i=1}^n v_i v_i^*$  the index of $E_C^A$, which we recall is central in $A$ and invertible.  We set
\begin{equation}
t= \sum_{i=1}^n \lambda(T^{-1/2}v_i)e_B\lambda(T^{-1/2}v_i^*),
\end{equation}
a symmetrised version of the element defined in \cite[Lemma 3.4]{Wat} and by a similar argument to the last displayed equation on page 5 of \cite{Wat} it follows that $t \in \lambda(A)'$ since $T^{-1/2}$ is central in $A$.
The row
\begin{equation}
M:=   (T^{-1/2}  v_1, \dots, T^{-1/2}  v_n)
\end{equation}
is in the unit ball of $M_{1,n}(A)$ since $MM^* = 1_A $ . 
By modifying the estimates in the last displayed equation on page 6 of \cite{Wat} to work with rows we have
\begin{align}\label{projest}
 \| t - e_B \|&=  \| \sum^n_{i=1} \lambda( T^{-1/2}v_i)(e_B \lambda (T^{-1/2}v_i^*) -\lambda ( T^{-1/2}v_i^*)e_B) \|  \nonumber \\
 &  =  \|\lambda^{(1,n)}(M) (\text{diag}^{(n)}(e_B) \lambda^{(n,1)}(M^*)-\lambda^{(n,1)}(M^*)e_B)\|  \nonumber \\
& \leq  \|\text{diag}^{(n)}(e_B) \lambda^{(n,1)}(M^*)-\lambda^{(n,1)}(M^*)e_B  \| \nonumber \\
& =  \|\text{diag}^{(n)}(e_B) \lambda^{(n,1)}(M^*)(1_\mathcal{E} -e_B)  - \text{diag}^{(n)}(1_\mathcal{E} -e_B)\lambda^{(n,1)}(M^*) e_B  \| \nonumber \\
& =  \max  \{  \| \text{diag}^{(n)}(e_B) \lambda^{(n,1)}(M^*)(1_\mathcal{E} -e_B) \lambda^{(1,n)}(M)\text{diag}^{(n)}(e_B)  \|^{\frac{1}{2} }, \nonumber \\
 &   \qquad \qquad \| e_B \lambda^{(1,n)}(M)\text{diag}^{(n)}(1_\mathcal{E} -e_B)\lambda^{(n,1)}(M^*)e_B  \|^{\frac{1}{2}} \} \leq 2\gamma^{\frac{1}{2}},
\end{align}
where the final bound is obtained by applying Lemma \ref{watlem2} and Lemma \ref{watlem} to each expression. We are now in a position to closely follow \cite[Lemma 3.4]{Wat} for the rest of the proof.
Set $\delta =  2\gamma^{\frac{1}{2}}$ so by \eqref{projest} and by following the argument in \cite[Lemma 3.4]{Wat} we may find a projection $ q \in  \lambda(A)' \cap C^*(\lambda(A), e_B, 1_{\mathcal{E}}) $ with
\begin{equation}
\|q - e_B \| \leq  2 \delta < 1.
\end{equation}
and a unitary $w \in C^*( \lambda(A), e_B, 1_{\mathcal{E}}) $ such that  $wqw^* = e_B$ and 
\begin{equation}
\|w- 1_{\mathcal{E}}\| \leq 2 \sqrt{2} \delta.
\end{equation}
By the choice of $q$ and $w$ the map $\tilde{\phi}: A \rightarrow \lambda(B)e_B$ defined for $x \in A$  by
\begin{equation}
\tilde{\phi} (x) = w q \lambda(x) q w^* = e_B w \lambda(x) w^* e_B
\end{equation}
 is a *-homomorphism. 
The map $\theta: B \rightarrow \lambda(B)e_B$ defined by $b \mapsto \lambda(b)e_B=e_B\lambda(b)e_B$ is a*-isomorphism so  $\phi := \theta^{-1} \circ \tilde{\phi}: A \rightarrow B$ is *-homomorphism which satisfies $\phi(c) = c$ (see \cite[Lemma 3.4]{Wat}). 

For $m,n \in \mathbb{N}$ and $x $ in the unit ball of $ M_{m,n}(A)$
\begin{align}
\| \phi^{(m,n)}(x) - E_B^{(m,n)}(x)\|  &= \|\text{diag}^{(m)}(e_B w) \lambda^{(m,n)}(x)\text{diag}^{(n)}(w^*e_B) -\text{diag}^{(m)}(e_B)  \lambda^{(m,n)}(x)\text{diag}^{(n)}(e_B) \| \\ &\leq  2\|1_\mathcal{E} - w\| \leq 4\sqrt{2} \delta.
\end{align}
Thus $\|\phi - E_B\|_{cb} \leq 4\sqrt{2} \delta$ hence, by Lemma \ref{watlem} we have the following estimate
\begin{equation}
\| \phi - \iota_A \|_{\text{row}} \leq 4 \sqrt{2} \delta + 2 \gamma =  8 \sqrt{2} \gamma^{\frac{1}{2}} + 2\gamma . \qedhere
\end{equation}
\end{proof}
We now modify \cite[Lemma 3.5]{Wat}, again working with the row norm and obtaining universal constants.

\begin{lem}\label{intertwine}
Let $C \subseteq D$  be a unital inclusion of C*-algebras and suppose $A \subseteq D$ is a C*-subalgebra containing $C$ with a finite index conditional expectation $E_C^A: A \rightarrow C$. Let  $\phi_1, \phi_2:  A \rightarrow D$ be unital *-homomorphisms such that $\phi_1|_C = \text{id}_C = \phi_2|_C$ and there exists a constant $\gamma$ such that $0 \leq \gamma < 1$ and $\| \phi_1 - \phi_2 \|_{\text{row}} \leq \gamma $. Then there exists a unitary $u \in  C' \cap D$ such that $\text{Ad}(u) \circ \phi_1 = \phi_2$ and $\|1 - u \| \leq 2 \gamma $, in particular, $\| \phi_1 - \phi_2 \|_{\text{cb}} \leq 4\gamma$.
\end{lem}
\begin{proof}
Let $(v_i)_{i=1}^n$ be a  quasi-basis for $E_C^A$ and $T$ be the index. As above, we symmetrise the element defined in \cite[Lemma 3.5]{Wat} and set
\begin{equation}
s= \sum_{i= 1}^n \phi_1(T^{-1/2}v_i) \phi_2(T^{-1/2}v_i^*).
\end{equation}
By a calculation similar to the second equation block in the proof of \cite[Lemma 3.5]{Wat} we have $\phi_1(a) s=  s \phi_2(a)$ for all $a \in A$.

As in the previous lemma, the row  
\begin{equation}
M:=  (T^{-1/2}  v_1, \dots, T^{-1/2}v_n)
\end{equation}
is in the unit ball of $ M_{1,n}(A)$.
Since $\phi_1$ is unital, using the row norm estimate in the hypothesis we have
\begin{align}
\|1 - s \| & =  \|  \sum_{i= 1}^n \phi_1(T^{-1/2}v_i)( \phi_1(T^{-1/2}v_i^*) - \phi_2(T^{-1/2}v_i^*)) \| \nonumber \\
               & = \|\phi_1^{(1,n)}(M)(\phi_1^{(n,1)}(M^*) - \phi_2^{(n,1)}(M^*)) \| \leq \gamma < 1
\end{align}
and hence $s $ is invertible in $D$.  The  polar decomposition $s = u |s |$ has unitary $u$ such that $\|1-u \| \leq \sqrt{2} \gamma$.
As in \cite[Lemma 3.5]{Wat} it follows that $\phi_1(a) = u \phi_2(a) u^*$ for all $a \in A$.
\qedhere
\end{proof}
Finally, we turn to the proof of our version of \cite[Proposition 3.6]{Wat}. 
\begin{thm}\label{variant}
Let $C \subseteq D$ be a unital inclusion of C*-algebras. Suppose that $B \subseteq D$ is a C*-algebra containing $C$ such that there exists a faithful conditional expectation $E_B^D: D \rightarrow B$. Suppose that $A \subseteq D$ is another C*-algebra  containing $C$ with a finite index conditional expectation $E_C^A: A \rightarrow C$ such that $d(A,B) \leq \gamma < 10^{-6}$. Then there exists a unitary $u \in C' \cap D$ such that $uAu^* = B$ with bound $\| u - 1\| \leq 16 \sqrt{110} \gamma^{\frac{1}{2}} + 880\gamma$.
\end{thm}

\begin{proof}
We set $\gamma' = 220 \gamma$ so Theorem \ref{rm} implies $d_{\text{row}}(A,B) \leq \gamma' < 1/2066$ so the hypothesis of Lemma \ref{homo} are satisfied. Hence there exists a *-homomorphism $ \phi: A \rightarrow B$ with  $\phi|_C = \text{id}_c$ such that
\begin{equation}
\| \phi- \iota_{A}\|_{\text{row}} \leq 8 \sqrt{2} \gamma'^{\frac{1}{2}} + 2\gamma'  <1
\end{equation}
by the choice of $\gamma'$. We apply Lemma \ref{intertwine} to the *-homomorphisms $\phi$ and $\iota_A$ to yield a unitary $u \in C^*(A,B)$ such that $\phi = \text{Ad}(u)$, in particular we have $u A u^* \subseteq B$,  and 
\begin{equation}\label{unitarybound}
\|1 - u \| \leq  16 \sqrt{2} \gamma'^{\frac{1}{2}} + 4\gamma' = 16 \sqrt{110} \gamma^{\frac{1}{2}} + 880\gamma.
\end{equation}
Let $b \in B_1$, we may find an element $a \in A_1$ such that $\| a - b\| \leq \gamma$. Applying the triangle inequality and using the bound \eqref{unitarybound} we compute
\begin{align}
\| u a u^* - b \| & \leq  \| (u- 1) au^* \| + \| a (u^* - 1) \| + \| a- b\| \nonumber \\
                          & \leq 32 \sqrt{2} \gamma'^{\frac{1}{2}} + 8\gamma'+\gamma < 1
\end{align}
by the choice of $\gamma$. Since $b \in B_1$ was arbitrary and $\|uau^* \|= \|a\| \leq 1$ we have $d(uAu^* , B) < 1$ and so it follows from a standard argument (see \cite[Proposition 2.4]{Chris10}) that $uAu^* = B$.
\end{proof}
\end{section}
\renewcommand{\abstractname}{Acknowledgements}
\begin{abstract}
 I would like to thank my PhD. supervisor, Stuart White, for his invaluable insight and guidance during this project. I would also like to express my gratitude to Vern Paulsen for his stimulating conversation at NBFAS Belfast, Roger Smith for his helpful comments and, as mentioned above, Erik Christensen for improving bounds in Lemma \ref{rowbound}.
\end{abstract}

\bibliography{tau}{}
\bibliographystyle{abbrv}
\end{document}